\documentclass{amsart}[12pt]
\usepackage{amssymb,amsmath}
\usepackage{enumerate}
\usepackage{color}

\newtheorem{theorem}{Theorem}
\newtheorem{proposition}{Proposition}
\newtheorem{corollary}[theorem]{Corollary}

\title{Isometries of perfect  norm  ideals of compact operators}

\begin{document}
\date{March 4, 2017}
\begin{abstract}
It is proved that for every surjective linear isometry $V$ on a perfect Banach symmetric ideal $\mathcal C_E \neq \mathcal C_{2}$  of compact operators, acting in a complex separable infinite-dimensional Hilbert $\mathcal H$ there exist unitary operators $u$ and $v$ on $\mathcal H$ such that $V(x) = uxv$ \ or \ $V(x) = ux^tv$
for all $x \in \mathcal C_E$, where $x^t $ is a transpose of an operator $x$ with respect to a fixed orthonormal basis for $\mathcal H$. In addition, it is shown that any  surjective 2-local isometry on a  perfect Banach symmetric ideal  $\mathcal C_E \neq \mathcal C_{2}$ is a linear isometry on  $\mathcal C_E$.
\end{abstract}

\author{Behzod Aminov}
\address{National University of Uzbekistan\\ Tashkent,  700174, Uzbekistan}
\email{behzod2@list.ru, aminovbehzod@gmail.com}

\author{Vladimir Chilin}
\address{National University of Uzbekistan\\ Tashkent,  700174, Uzbekistan}
\email{ vladimirchil@gmail.com, chilin@ucd.uz}

\maketitle

\section{Introduction}

Let  $\mathcal H$ be a complex separable infinite-dimensional Hilbert space. Let   $(E,\|\cdot\|_E) \subset c_0$  be a real  Banach symmetric sequence space. Consider an ideal $\mathcal C_E$  of compact linear operators in $\mathcal H$,  which is defined by the relations
$$ x\in \mathcal C_E \Longleftrightarrow \{s_n(x)\}_{n=1}^{\infty}\in E$$
and
$$ \|x\|_{\mathcal C_E}=\|\{s_n(x)\}_{n=1}^{\infty}\|_E,$$
where $\{s_n(x)\}_{n=1}^{\infty}$ are the singular values of $x$ (i.e. the eigenvalues of $(x^*x)^{1/2}$ in decreasing order). In the paper \cite{KS} it is shown  that $\|\cdot\|_{\mathcal C_E}$ is a Banach norm on  $\mathcal C_E$. In addition,  $\mathcal C_1:=\mathcal C_{l_1} \subset \mathcal C_E.$

 In the case when $(E,\|\cdot\|_E)$ is a separable Banach space, the Banach  ideal $(\mathcal C_E,\|\cdot\|_{\mathcal C_E})$ is a minimal Banach ideal in the terminology of Schatten \cite{S}, i.e. the set of finite rank operators is dense  in $\mathcal C_E$.

It is known \cite{sourour} that for every surjective linear isometry $V$ on a minimal Banach ideal  $\mathcal C_E \neq \mathcal C_{2}$, where $\mathcal C_{2}:=\mathcal C_{l_2}$, there exist unitary operators $u$ and $v$ on $\mathcal H$ such that
\begin{equation}\label{e1}
V(x) = uxv \ \ (\text{or} \ \ V(x) = ux^tv)
\end{equation}
for all $x \in \mathcal C_E$, where $x^t $ is the transpose of the operator $x$ with respect to a fixed orthonormal basis in $\mathcal H$. In the case of a Banach  ideal $\mathcal C_{1}$,  a description of surjective linear isometries  of the form of (\ref{e1}) was obtained in \cite{r}, and  for the ideals $\mathcal C_{p}:= \mathcal C_{l_p}, \ 1 \leq p \leq \infty, \ p\neq 2$, in \cite{a}.

In this paper we show that, in the case  a  Banach symmetric sequence space $E\neq l_2$  with   Fatou property, every surjective linear isometry $V$ on $\mathcal C_E$ has the form (1). In addition, it is proved that in this case any 2-local surjective isometry on $\mathcal C_E$ also is of the form (1).

\section{Preliminaries}
Let $l_{\infty}$ (respectively, $c_0$)  be a Banach lattice of all bounded (respectively, converging to zero) sequences $\{ \xi_n \}_{n = 1}^{\infty}$  of
 real numbers with respect to the norm $\|\{\xi_n\}\|_{\infty} = \sup\limits_{n \in \mathbb N} |\xi_n|$, where $\mathbb N $ is the set of natural numbers.
If  $ x = \{\xi_n\}_{n = 1}^{\infty} \in l_{\infty}$, then a non-increasing rearrangement of $x$  is defined by
$$
x^* = \{ \xi_n^* \}, \ \ \text{where}  \ \ \ \xi_n^*: = \inf\limits_{| F | <n}\sup\limits_{n \notin F} |\xi_n |, \ \ F \  \text{is a finite subset of} \ \ \mathbb N.
$$
A non-zero linear subspace  $E \subset l_{\infty}$ with a Banach norm $\|\cdot\|_{E}$ is called {\it Banach symmetric sequences  space} if the conditions  $ y \in E$, $ x \in l_{\infty}, \  x^* \leq y^* $, imply that $ x \in E $ and $ \| x \|_E \leq \| y\|_E $. In this case, the  inequality $\| x \|_{\infty} \leq \|x\|_E $ follows  for each $ x \in E $.

Let $\mathcal H$ be a complex separable infinite-dimensional Hilbert space and let
 $\mathcal B(\mathcal H)$ (respectively,  $\mathcal K(\mathcal H), \mathcal F(\mathcal H)$) be the $\ast$-algebra of all bounded (respectively, compact, finite rank) linear operators in $\mathcal H$. It is well known
 that $$\mathcal F(\mathcal H) \subset \mathcal I\subset \mathcal K(\mathcal H)$$ for any proper two-sided ideal
$\mathcal I$ in $\mathcal B(\mathcal H)$ (see for example, \cite[Proposition 2.1]{SI}).

If $(E, \|\cdot\|_{E}) \subset c_0$ is a   Banach symmetric sequence  space, then the set
$$\mathcal{C}_E:=\{ x \in \mathcal K(\mathcal H) : \{s_n(x)\}_{n=1}^\infty \in E\}$$
is a proper two-sided ideal in $\mathcal B(\mathcal H)$ (\cite{C}; see also  \cite[Theorem 2.5]{SI}). In addition, $(\mathcal{C}_E, \|\cdot\|_{\mathcal{C}_E})$) is a Banach  space with respect to the norm $ \|x\|_{\mathcal C_E}=\|\{s_n(x)\}_{n=1}^{\infty}\|_E$ \cite{KS}, and the norm $ \|\cdot\|_{\mathcal C_E}$ has the following properties

1) $\|xzy\|_{\mathcal C_E} \leq
\|x\|_{\infty}\|y\|_{\infty}\|z\|_{\mathcal C_E}$ for all $x,y\in
\mathcal B(\mathcal H)$ and $z\in \mathcal C_E$, where $\|\cdot\|_{\infty}$ is the usual
operator norm in $\mathcal B(\mathcal H)$;

2) $\|x\|_{\mathcal C_{E}}=\|x\|_{\infty}$ if $x \in \mathcal C_E$  is of rank 1.

In this case we say that $(\mathcal{C}_E, \|\cdot\|_{\mathcal{C}_E})$) is a {\it Banach symmetric
ideal}  (cf. \cite[Ch. III]{gohberg}, \cite[Ch. 1, \S 1.7]{SI}). \
It is clear that
 $\|uxv\|_{\mathcal{C}_E} = \|x\|_{\mathcal{C}_E}$ for all unitary operators $u,v\in \mathcal B(\mathcal H)$ and $x\in {\mathcal{C}_E}$. Besides,
$$
\mathcal C_1 \subset \mathcal C_E \subset \mathcal K(\mathcal H) \ \ \text{and} \ \ \|x\|_{\mathcal C_E}\leq  \|x\|_{\mathcal C_1}, \ \ \|y\|_{\infty}\leq  \|y\|_{\mathcal C_E}
$$
for all $x \in \mathcal{C}_1, \ y \in \mathcal{C}_E$.

A Banach symmetric sequence space $(E, \|\cdot\|_{E})$ (respectively, a Banach symmetric ideal
$(\mathcal{C}_E, \|\cdot\|_{\mathcal{C}_E})$)   is said to have {\it order continuous norm} if $\| a_{k}\|_E\downarrow
0$ (respectively, $\| a_{k}\|_{\mathcal{C}_E}\downarrow 0$) whenever $a_{k}\in E$ (respectively, $a_{k}\in \mathcal{C}_E)$ and
$a_{k}\downarrow 0$.  It is known that $(\mathcal{C}_E, \|\cdot\|_{\mathcal{C}_E}))$ has order continuous norm if  and
only if $(E, \| \cdot \|_E)$  has order continuous norm (see, for example, \cite[Proposition 3.6]{DDP1}). Moreover,
every  Banach symmetric ideal with order continuous norm is  a minimal norm ideal, i.e. the subspace $\mathcal
F(\mathcal H)$  is dense  in $\mathcal C_E$.

If $(E, \|\cdot\|_{E})$ is a  Banach symmetric sequence space  (respectively, $(\mathcal{C}_E, \|\cdot\|_{\mathcal{C}_E})$ is  a Banach symmetric ideal), then the K\"{o}the  dual $E^\times$ (respectively, $\mathcal{C}_E^\times$) is defined as
$$
E^\times=\{\xi=\{\xi_n\}_{n=1}^{\infty} \in l_{\infty}\,: \ \, \xi\eta\in l_1 \ \ \text{for all} \ \ \eta\in E\}
$$
$$
(\text{respectively}, \ \ \mathcal{C}_E^\times=\{x\in B(H)\,: \ \, xy\in \mathcal{C}_1 \ \ \text{for all} \ \ y\in \mathcal{C}_E\})
$$
and
$$
\|\xi\|_{E^\times}=\sup\limits\{\sum\limits_{n=1}^{\infty}|\xi_n \eta_n|: \eta=\{\eta_n\}_{n=1}^{\infty} \in E, \ \|\eta\|_{E}\leq 1  \} \ \ \text{if} \ \ \xi \in E^\times
$$
$$
(\text{respectively}, \ \ \|x\|_{\mathcal{C}_E^\times}=\sup\limits\{tr(|xy|): y\in \mathcal{C}_E, \ \|y\|_{\mathcal{C}_E}\leq 1  \} \ \ \text{if} \ \ x \in \mathcal{C}_E^\times),
$$
where $tr(\cdot)$ is the standard trace on  $\mathcal B(\mathcal H)$.

It is known that $(E^\times, \|\cdot\|_{E^\times})$ (respectively, $(\mathcal{C}_E^\times,
\|\cdot\|_{\mathcal{C}_E^\times})$ is a  Banach symmetric sequence space (respectively, a  Banach symmetric ideal). In addition, $\mathcal{C}_1^\times = \mathcal B(\mathcal H)$ and if $\mathcal{C}_E \neq \mathcal{C}_1$, then
$\mathcal{C}_E^\times \subset \mathcal K(\mathcal H)$ \cite[Proposition 7]{garling}. We also note the following useful
property \cite[Theorem 5.6]{DDP1}:
$$
(\mathcal{C}_E^\times, \|\cdot\|_{\mathcal{C}_E^\times})=(\mathcal{C}_{E^\times}, \|\cdot\|_{\mathcal{C}_{E^\times}}).
$$
A  Banach symmetric ideal $\mathcal{C}_E$  is said to be {\it perfect} if $\mathcal{C}_E = \mathcal{C}_E^{\times\times}$. It
 is clear that $\mathcal{C}_E$  is perfect if and only if $E = E^{\times\times}$.

A Banach symmetric  sequence space $(E, \|\cdot\|_{E})$
 (respectively, a   Banach symmetric ideal $(\mathcal{C}_E, \|\cdot\|_{\mathcal{C}_E})$) is said to possess
{\it Fatou property} if the conditions
$$
 0 \leq a_{k }\leq a_{k+1}, \  a_{k}\in E \ \ \text{(respectively}, \ a_{k}\in \mathcal{C}_E) \ \ \text{for all} \ \ k \in \mathbb N
 $$
$$ \text{\ and\ } \ \ \sup_{k\geq 1} \| a_{k}\|_E<\infty \ \ \text{(respectively}, \ \sup_{k\geq 1} \| a_{k}\|_{\mathcal{C}_E}<\infty)
$$
imply that there exists $a \in E \ (a \in \mathcal{C}_E)$ such that $a_{k}\uparrow a$ and $\| a\|_E=\sup \limits_{k\geq
1} \| a_{k}\|_E$  (respectively, $\| a\|_{\mathcal{C}_E}=\sup \limits_{k\geq 1} \| a_{k}\|_{\mathcal{C}_E}$).

It is known that $(E, \|\cdot\|_{E})$ (respectively, $(\mathcal{C}_E, \|\cdot\|_{\mathcal{C}_E})$) has the Fatou
property if and only if $E = E^{\times\times}$ \cite[Vol. II, Ch. 1, Section a]{LT} (respectively,  $\mathcal{C}_E=
\mathcal{C}_E^{\times\times}$ \cite[Theorem 5.14]{DDP1}). Therefore
$$
\mathcal{C}_E  \ \ \text{is perfect} \ \Leftrightarrow \ \mathcal{C}_E=\mathcal{C}_E^{\times\times}\ \Leftrightarrow \ E = E^{\times\times} \ \Leftrightarrow
$$
$$
 E  \ \ \text{has the Fatou property} \ \Leftrightarrow \ \mathcal{C}_E  \ \ \text{has Fatou property}.
$$

If $y\in  \mathcal{C}_E^{\times}$, then a linear functional $f_y(x)=tr (xy)=tr (yx), \ x \in  \mathcal{C}_E$,  is continuous on  $ \mathcal{C}_E$. In addition, $\|f_y\|_{\mathcal{C}_E^{\ast}} = \|y\|_{\mathcal{C}_E^{\times}}$,   where $\mathcal{C}_E^*$ is the dual of the Banach space  $(\mathcal{C}_E, \|\cdot\|_{\mathcal{C}_E})$ (see, for example, \cite{garling}).
Identifying an element $y \in \mathcal{C}_E^{\times}$ and the linear functional $f_y$, we may assume that $\mathcal{C}_E^{\times}$ is a closed linear subspace in $ \mathcal{C}_E^*$.
Since $\mathcal F(\mathcal H) \subset \mathcal{C}_E^{\times}$, it follows that $\mathcal{C}_E^{\times}$ is a total subspace in $ \mathcal{C}_E^*$. Thus, the weak topology $\sigma(\mathcal{C}_E, \mathcal{C}_E^{\times})$ is a Hausdorff topology.

\begin{proposition}\label{p1}\cite[Proposition 2.8]{dodds}
A linear functional $f \in \mathcal{C}_E^*$ is continuous with respect to weak topology $\sigma(\mathcal{C}_E, \mathcal{C}_E^{\times})$ (i.e.  $f \in \mathcal{C}_E^\times$) if and only if $f(x_n)\to 0$ for every sequence $ x_n \downarrow 0, \ x_n \in \mathcal{C}_E, \ n \in \mathbb N$.
\end{proposition}

It is clear that $\mathcal{C}_E^h = \{x \in \mathcal{C}_E: x = x^{\ast}\}$ is a real Banach space with respect to the
norm $\|\cdot\|_{\mathcal{C}_E}$; besides, the cone $\mathcal{C}_E^+ = \{x \in \mathcal{C}_E^h: x \geq 0\}$ is
closed in $\mathcal{C}_E^h$ and $\mathcal{C}_E^h = \mathcal{C}_E^+ - \mathcal{C}_E^+$. If  $ 0 \leq x_n \leq \ x_{n+1}
\leq x, \ x, x_n \in \mathcal{C}_E^h, \ n \in \mathbb N$, then there exists $\sup_{n\geq 1} x_n \in \mathcal{C}_E^h$,
i.e. $\mathcal{C}_E^h$ is a quasi-$(o)$-complete space \cite{an}. Hence, by \cite[Theorem 1]{an}, we have the following
proposition.

\begin{proposition}\label{p2}
Every linear functional  $f\in (\mathcal{C}_E^h)^*$ is a difference
of two positive linear functionals from $(\mathcal{C}_E^h)^*$.
\end{proposition}

We need the following property  of the weak topology
$\sigma(\mathcal{C}_E, \mathcal{C}_E^{\times})$.

\begin{proposition}\label{p3}
If $x\in \mathcal{C}_E$, then there exists a sequence
$\{x_n\} \subset \mathcal F(\mathcal H)$ such that $x_n\stackrel{\sigma(\mathcal{C}_E,\mathcal{C}_E^{\times})}{\longrightarrow}
x$.
\end{proposition}

\begin{proof} It suffices to establish the validity of Proposition \ref{p3} for
$$0 \leq x = \sum\limits_{j=1}^{\infty} \lambda_j p_j \in \mathcal{C}_E \setminus \mathcal F(\mathcal H),
$$
where $\lambda_j\ge 0$ are the eigenvalues of the compact operator $x$ and $p_j\in \mathcal
F(\mathcal H)$  are  finite-dimensional projectors for all $j\in\mathbb N$ (the
series converges with respect to the norm
$\|\cdot\|_{\infty}$).
 If $y_n =\sum\limits_{j=n}^{\infty} \lambda_j
p_j$, then  $y_n \downarrow 0$ and, by Proposition, \ref{p1} we have $f(y_n)\to 0$ for all $f\in \mathcal{C}_E^{\times}$. Therefore
$$
\mathcal F(\mathcal H) \ni x-y_n\stackrel{\sigma(\mathcal{C}_E,\mathcal{C}_E^{\times})}{\longrightarrow} x.
$$
\end{proof}

Let $T$ be a bounded linear operator acting in a  Banach symmetric ideal  $(\mathcal{C}_E,\|\cdot\|_{\mathcal{C}_E})$,
and let $T^*$ be its adjoint operator.
If  $T^*(\mathcal{C}_E^\times)\subset \mathcal{C}_E^\times$, then
$$
f(T(x_{\alpha}))=T^{*}(f)(x_{\alpha})\to T^{*} (f)(x)=f(T(x))
$$
for every  net  $x_\alpha \stackrel{\sigma(\mathcal{C}_E,\mathcal{C}_E^{\times})}{\longrightarrow}
x , \ x_{\alpha}, x \in \mathcal{C}_E$, and for all $f \in \mathcal{C}_E^\times$.
Thus  $T$ is a $\sigma(\mathcal{C}_E,\mathcal{C}_E^\times)$-continuous operator.

Inversely, if the operator $T$ is a $\sigma(\mathcal{C}_E,\mathcal{C}_E^{\times})$-continuous and
 $f\in \mathcal{C}_E^{\times}$, then the linear functional $(T^*f)(x)=f(T(x))$
is also $\sigma(\mathcal{C}_E,\mathcal{C}_E^{\times})$-continuous, i.e. $T^*(f)\in \mathcal{C}_E^{\times}$.

Therefore, a linear bounded operator $T\colon \mathcal{C}_E\to \mathcal{C}_E$ is
$\sigma(\mathcal{C}_E,\mathcal{C}_E^{\times})$-continuous  if and only if $T^*(f)\in \mathcal{C}_E^{\times}$ for every
$f\in \mathcal{C}_E^{\times}$.
Thus, for any $\lambda\in\mathbb C$ and linear
bounded operators  $T,S:  \mathcal{C}_E \to  \mathcal{C}_E$,
continuous with respect to a $\sigma(\mathcal{C}_E,\mathcal{C}_E^{\times})$-topology, the operators $T+S$, $TS$ and  $\lambda T$ are $\sigma(\mathcal{C}_E,\mathcal{C}_E^{\times})$-continuous.

Let  $\mathcal B(\mathcal{C}_E)$ be the Banach
space of bounded linear operators $T$ acting in $(\mathcal{C}_E, \|\cdot\|_{\mathcal{C}_E})$ with the norm
$$\|T\|_{\mathcal B(\mathcal{C}_E)} = \sup\limits_{\|x\|_{\mathcal{C}_E} \leq 1} \|T(x)\|_{\mathcal{C}_E}.$$

\begin{proposition}\label{p4}
If $T_n\in \mathcal B(\mathcal{C}_E)$ are
$\sigma(\mathcal{C}_E,\mathcal{C}_E^{\times})$-continuous operators, $T\in
\mathcal B(\mathcal{C}_E)$, and $\|T_n-T\|_{\mathcal
B(\mathcal{C}_E)}\to 0$ as $n\to\infty$, then  $T$ is also
$\sigma(\mathcal{C}_E,\mathcal{C}_E^{\times})$-continuous.
\end{proposition}

\begin{proof} It suffices to show that $T^*(f)\in \mathcal{C}_E^{\times}$ for any functional $f\in \mathcal{C}_E^{\times}$.
Since $T_n\in \mathcal B(\mathcal{C}_E)$ are
$\sigma(\mathcal{C}_E,\mathcal{C}_E^{\times})$-continuous operators, it follows that \
$T_n^*(f)\in \mathcal{C}_E^{\times}$ \ for all  \ $f\in
\mathcal{C}_E^{\times}, \ n \in \mathbb N$. Considering the  closed subspace
$\mathcal{C}_E^{\times}$  in $(\mathcal{C}_E^*,\|\cdot\|_{\mathcal{C}_E^*})$ and noting that
$$
\|T_n^*(f)-T^*(f)\|_{\mathcal{C}_E^*}\leq \|T_n^*-T^*\|_{\mathcal B(\mathcal{C}_E^*)}
\|f\|_{\mathcal{C}_E^*}=\|T_n-T\|_{\mathcal B(\mathcal{C}_E)}\|f\|_{\mathcal{C}_E^*}\to 0,
$$
we conclude that $T^*(f)\in \mathcal{C}_E^{\times}$.
\end{proof}

We also need the following well-known properties of  perfect symmetrically normed ideals.

\begin{theorem}\label{t3}
Let $(\mathcal{C}_E, \|\cdot\|_{\mathcal{C}_E}) \subset c_0$ be a Banach symmetric sequence space with Fatou property. Then

$(i)$. \cite[Theorem 5.11]{DDP1}.
$
\|x\|_{ \mathcal{C}_E} = \sup\limits_{y\in
\mathcal{C}_E^{\times}, \|y\|_{\mathcal{C}_E^{\times}} \leq 1} |tr(xy)|
$
\ for every  \ $x\in  \mathcal{C}_E$;

$(ii)$. \cite[Theorem 3.5]{dodds}. A Banach space  $\mathcal{C}_E$ is
$\sigma(\mathcal{C}_E,\mathcal{C}_E^{\times})$-sequentially
complete, i.e. if for $x_n \in \mathcal{C}_E, \ n \in \mathbb N$, and for every $f\in \mathcal{C}_E^{\times}$ there is $\lim\limits_{n\to\infty} f(x_n)$, then there exists $x\in
\mathcal{C}_E$ such that
$x_n\stackrel{\sigma(\mathcal{C}_E,\mathcal{C}_E^{\times})}{\longrightarrow} x$.
\end{theorem}

\section{The ball topology in  norm  ideals of compact operators}

Let  $(X,\|\cdot\|_X)$ be a real Banach space and let $b_X$ be the ball topology in  $X$, i.e. $b_X$ is the coarsest topology
 such that every closed ball $$B(x,\varepsilon) =\{y \in X: \|y-x\|_X\leq \varepsilon\}, \ \varepsilon > 0,$$
 is  closed in  $b_X$ \cite{godefroykalton}. The family
$$
X \setminus \bigcup_{i=1}^n B(x_i,\varepsilon_i), \ x_i \in X, \ \|x_0-x\|_X > \varepsilon_i, \ i=1,...,n, \ n \in \mathbb N.
$$
is a base of neighborhoods of the point $x_0 \in X$ in $b_X$.
Therefore   $x_\alpha\stackrel{b_X}{\longrightarrow} x, \ x_\alpha, x \in X$, if and only if  $\lim\inf
\|x_\alpha-x\|_X\geq \|x-y\|_X$ for all $y\in X$ \cite{godefroykalton}. In particular,  every surjective isometry   $V$
in $(X,\|\cdot\|_X)$ is continuous with respect to the ball topology $b_X$.

Let us note that $b_X$ is not a Hausdorff topology. The following theorem provides   a sufficient condition for $T_2$-axiom of
$b_X$ on subsets of $X$.  Recall that $A$ is a Rosenthal subset of $(X,\|\cdot\|_X)$   if every sequence in $A$ has a weakly Cauchy subsequence.

\begin{theorem}\cite[Theorem 3.3]{godefroykalton}.\label{t4}
Let  $(X,\|\cdot\|_X)$ be a real Banach space, and let  $A$ be a bounded absolutely convex Rosenthal subset of $X$. Then $(A, b_X)$ is a Hausdorff space.
\end{theorem}

In the proof of required properties of the ball topology (see
Theorem \ref {t5} below), we utilize the following well-known proposition.

\begin{proposition}\cite[section 3, Ch. 1, \S 5]{polya}. \label{p5}
Let $p_{nk}$ be real numbers, $n,k\in\mathbb N$, such that
$\sum\limits_{k=1}^n |p_{nk}| = 1$ for all $n\in\mathbb N$, and  let the limit $\lim\limits_{n\to\infty} p_{nk}=p_k$
exist for every fixed $k \in\mathbb N$. Then the sequence
$$
s_n=p_{n1}r_1+p_{n2}r_2+\ldots + p_{nn}r_n
$$
converges for every convergent sequence  $\{r_n\}$.
\end{proposition}

\begin{theorem}\label{t5}
Let $(\mathcal{C}_E, \|\cdot\|_{\mathcal{C}_E})$ be a  Banach symmetric ideal, $x_n\in \mathcal{C}_E^+$, and let
$x_n\downarrow 0$. Then the sequence  $\{x_n\}$ can converge with respect to the topology $b_{\mathcal{C}_E}$
to no more than one element.
\end{theorem}

\begin{proof} Consider $\mathcal{C}_E$ as a real Banach space. Let $A$ be the absolutely convex hull of the sequence
$\{0, x_1, ..., x_n,...\}$. Since  $x_n\downarrow 0$ and the norm $\|\cdot\|_{\mathcal{C}_E}$ is monotone,
it follows that $A$ is a bounded subset of  $\mathcal{C}_E$. Let us show that  $A$ is a Rosenthal subset.

Using the inequality   $0 \leq x_{n+1} \leq x_n$ and a decomposition of
linear functional $f\in (\mathcal{C}_E^h)^*$ as the difference
of two positive functionals from $(\mathcal{C}_E^h)^*$  (see Proposition \ref{p2}), we conclude
that there exists $\lim\limits_{n\to\infty} f(x_n)=r_n$  for every functional
$f\in (\mathcal{C}_E^h)^*$.
For any sequence $\{y_n\}_{n=1}^{\infty}\subset A$, we have
$$
y_n=p_{n1}x_1+p_{n2}x_2+\ldots +p_{nk(n)}x_{k(n)}, \ \ \text{where}
\ \ \sum\limits_{i=1}^{k(n)} |p_{ni}| = 1;
$$
in particular, $p_{ni}\in [-1,1]$ for all $i\in 1,\ldots , k(n)$. Consider a sequence
$$
q_n=(p_{n1},p_{n2}\ldots , p_{nk(n)}, 0,0\ldots )\in
\prod\limits_{i=1}^{\infty} [-1,1].
$$
By Tychonoff Theorem,  a set
$\prod\limits_{i=1}^{\infty} [-1,1]$ is compact with respect to the product
topology. Moreover, by \cite[ch. 4, theorem 17]{kelli}, the
set $\prod\limits_{i=1}^{\infty} [-1,1]$ is a metrizable compact. Hence, the sequence
$\{q_n\}$ has a convergent subsequence
$\{q_{n_i}\}_{i=1}^{\infty}$. In particular, there are the limits
$\lim\limits_{n_i\to\infty} p_{n_i k}$ for all $k\in\mathbb N$. Besides, the sequence $f(y_{n_i})=\sum\limits_{j=1}^k
p_{{n_i}j} f(x_j), \ f\in  (\mathcal{C}_E^h)^*$,  satisfies all conditions of
Proposition \ref{p5},  which implies its convergence.
Therefore, the sequence
$\{y_n\}_{n=1}^{\infty}\subset A$  has a weakly Cauchy subsequence $\{y_{n_i}\}$.
This means that $A$ is a Rosenthal subset  of  $(\mathcal{C}_E^h, \|\cdot\|_{\mathcal{C}_E})$.

By Theorem \ref{t4}, the topological space $(A,b_{\mathcal{C}_E^h})$ is
Hausdorff, hence the sequence  $\{x_n\}_{n=1}^{\infty}\subset
A$ could not have more than one limit with respect to the topology $b_{\mathcal{C}_E^h}$.
Since $\mathcal{C}_E^h$ is a closed subspace in $(\mathcal{C}_E, \|\cdot\|_{\mathcal{C}_E})$ (we assume that $\mathcal{C}_E$ is a real space), it follows that the restriction $(b_{\mathcal{C}_E}) | _{\mathcal{C}_E^h} $ is finer than $b_{\mathcal{C}_E^h}$. Therefore, the sequence $\{x_n\}_{n=1}^{\infty}$  can have
no more than one limit with respect to the topology $b_{\mathcal{C}_E}$.
\end{proof}

\begin{proposition}\label{p6}
If $E \subset c_0$ is a  Banach symmetric sequence space  with Fatou property,
 then $b_{\mathcal{C}_E} \leq \sigma(\mathcal{C}_E,\mathcal{C}_E^{\times})$.
\end{proposition}

\begin{proof}  It suffices to show that  $B(0,1) =\{x \in \mathcal{C}_E: \|x\|_{\mathcal{C}_E}\leq 1\}$
is closed with respect to the weak
topology $\sigma(\mathcal{C}_E,\mathcal{C}_E^{\times})$. Let $x_\alpha\in B(0,1)$ \ and \
$x_\alpha\stackrel{\sigma(\mathcal{C}_E,\mathcal{C}_E^{\times})}{\longrightarrow} x\in \mathcal{C}_E$. Assume that
$x\notin B(0,1)$, i.e. $\|x\|_{\mathcal{C}_E}=q > 1$, and let $\epsilon > 0$ be such that  $q-\epsilon > 1$. By Theorem
\ref{t3} $(i)$, there exists  $y\in \mathcal{C}_E^{\times}$ such that  $\|y\|_{\mathcal{C}_E^{\times}}\leq 1$ and
 $q\geq |tr(xy)| > q-\epsilon$.

On the other hand,
$x_n\stackrel{\sigma(\mathcal{C}_E,\mathcal{C}_E^{\times})}{\longrightarrow} x$ implies that
$|tr(x_ny)| \to |tr(xy)|$. Since $x_n\in B(0,1)$, it follows that
 $|tr(x_ny)| \leq 1$ and therefore $|tr(xy)| \leq 1$, which is
impossible.
Thus   $B(0,1)$  is a closed set in $\sigma(\mathcal{C}_E,\mathcal{C}_E^{\times})$.
\end{proof}

\section{Weak continuity of Hermitian operators in the perfect ideals  of compact operators}

In this section, we establish $\sigma(\mathcal{C}_E,\mathcal{C}_E^{\times})$-continuity of the Hermitian operator acting in a perfect  Banach symmetric ideal $\mathcal{C}_E$ of compact operators.
For that we need $\sigma(\mathcal{C}_E,\mathcal{C}_E^{\times})$-continuity of a surjective isometry on $\mathcal{C}_E$.

Recall that the series $\sum\limits_{n=1}^{\infty}x_{n}$  converges weakly  unconditionally in a Banach space
$X$ if a numerical series $\sum\limits_{n=1}^{\infty} f(x_{n})$   converges absolutely for every $f\in X^*$ \cite[Ch.
2, \S 3]{woj}.

\begin{proposition} \cite[Ch. 2, \S 3]{woj}.\label{p7}
Let $(X,\|\cdot\|_X)$  be a Banach space, $x_n \in X, \ n \in \mathbb N$. Then the following conditions are equivalent: \\
$(i)$ A series $\sum\limits_{n=1}^{\infty}x_{n}$ converges weakly  unconditionally; \\
$(ii)$ There exists a constant  $C>0$, such that
$$
\sup\limits_N \|\sum\limits_{n=1}^N t_nx_{n} \|_X \leq C
\|\{t_n\}_{n=1}^{\infty}\|_{\infty}
$$
for all $\{t_n\}_{n=1}^{\infty}\in l_{\infty}.$
\end{proposition}

Proposition \ref{p7} implies the following.
\begin{corollary}\label{c1}
If $V$ is a surjective linear isometry on a Banach space  $X$ and a series $\sum\limits_{n=1}^{\infty}x_{n}$ converges
weakly  unconditionally in  $X$, then the series $\sum\limits_{n=1}^{\infty}V(x_n)$ also converges weakly unconditionally in  $X$.
\end{corollary}
Now we can show that every surjective linear isometry of $\mathcal{C}_E$ is $\sigma(\mathcal{C}_E,\mathcal{C}_E^{\times})$-continuous.
\begin{proposition}\label{p8}
Let  $\mathcal{C}_E$ be a perfect  Banach symmetric ideal of compact operators and $V$
a surjective linear isometry on $\mathcal{C}_E$. Then
$V$ is $\sigma(\mathcal{C}_E,\mathcal{C}_E^{\times})$-continuous.
\end{proposition}

\begin{proof} It suffices to show that $ V^*(f) \in \mathcal{C}_E^{\times}$ for each functional
  $ f \in \mathcal{C}_E^{\times}$. According to Proposition \ref{p1}, it should be established that
$x_n\downarrow 0, \  \{x_n\}_{n=1}^{\infty}\subset  \mathcal{C}_E^+$, implies $f(V (x_n)) = V^*(f)(x_n) \to 0 $.

Let  $\{x_n\}_{n=1}^{\infty}\subset \mathcal{C}_E $ \ and \ $x_n\downarrow 0$. We will show that the series
$\sum\limits_{n=1}^{\infty} (x_n-x_{n+1})$ \  converges weakly  unconditionally in  $\mathcal{C}_E$. If $f\in
(\mathcal{C}_E^h)^*$ is a positive linear functional, then
$$
\sum\limits_{n=1}^m |f(x_n-x_{n+1})|=f(\sum\limits_{n=1}^m
(x_n-x_{n+1}))=f(x_1-x_{m+1})\leq f(x_1)
$$
for all $m\in\mathbb N$. Hence the numerical series $\sum\limits_{n=1}^{\infty} f(x_n-x_{n+1})$ converges absolutely.
Since every functional $f\in (\mathcal{C}_E^h)^*$ is the difference of two positive functionals from
$(\mathcal{C}_E^h)^*$ (see Proposition \ref{p2}), the series $\sum\limits_{n=1}^{\infty} (x_n-x_{n+1})$
converges weakly  unconditionally in  $\mathcal{C}_E^h$.

Let $f\in \mathcal{C}_E^*$. Denote
$$
u(x)= Re f(x) = \frac{f(x) + \overline{f(x)}}{2}, \ v(x)= Im f(x) = \frac{f(x) - \overline{f(x)}}{2i}, \ x \in \mathcal{C}_E.
$$
It is clear that $u, v \in (\mathcal{C}_E^h)^*$. Therefore, the series $\sum\limits_{n=1}^{\infty} u(x_n-x_{n+1})$ and $\sum\limits_{n=1}^{\infty} v(x_n-x_{n+1})$ converge absolutely. Thus, the series $\sum\limits_{n=1}^{\infty} f(x_n-x_{n+1})$ also converges absolutely. This means that a series $\sum\limits_{n=1}^{\infty} (x_n-x_{n+1})$ converges weakly
unconditionally in  $\mathcal{C}_E$.

By Corollary \ref{c1}, the series $$\sum\limits_{n=1}^{\infty} (Vx_n-Vx_{n+1})$$ converges weakly unconditionally in  $\mathcal{C}_E$. Hence the numerical series
$$
\sum\limits_{n=1}^{\infty}
f(Vx_n-Vx_{n+1})=f(Vx_1)-\lim\limits_{n\to\infty}f(Vx_n)
$$
converges for every $f\in \mathcal{C}_E^*$. In particular, the limit
 $\lim\limits_{n\to\infty} f(Vx_n)$ exists. Since  $ \mathcal{C}_E$ is a
$\sigma(\mathcal{C}_E,\mathcal{C}_E^{\times})$-sequentially complete set (see Theorem \ref{t3}$(ii)$)
and  $V$  is a bijection, it follows that there exists  $x_0\in \mathcal{C}_E$ such that
$V(x_n)\stackrel{\sigma(\mathcal{C}_E,\mathcal{C}_E^{\times})}{\longrightarrow} V(x_0)$ as
$n\to\infty$.

It remains to show that $V(x_0)=0$. By Proposition \ref{p6}, $$V(x_n)\stackrel{b_{\mathcal{C}_E}}{\longrightarrow} V(x_0).$$
Since the isometry $V^{-1}$ is continuous with respect to the topology  $b_{\mathcal{C}_E}$, it follows that $x_n\stackrel{b_{\mathcal{C}_E}}{\longrightarrow} x_0$.  Now, taking into account that $x_n\downarrow 0$, we obtain $x_n\stackrel{\sigma(\mathcal{C}_E,\mathcal{C}_E^{\times})}{\longrightarrow} 0$ (see Proposition \ref{p1}).
Then, by Proposition \ref{p6}, $x_n\stackrel{b_{\mathcal{C}_E}}{\longrightarrow} 0$, and
Theorem \ref{t5} implies that $x_0=0$.
\end{proof}

Let $(X,\|\cdot\|_X)$ be a complex Banach space.
A bounded linear operator  $T\colon X\to X$  is called
Hermitian, if the operator \ $e^{itT}=\sum\limits_{n=0}^{\infty}
\frac{(itT)^n}{n!}$ \ is an isometry of the space $X$ for all
$t\in\mathbb R$ (see, for example, \cite[Ch. 5, \S 2]{flem}).

We show that a Hermitian operator acting in a perfect  Banach symmetric ideal $\mathcal{C}_E$ is $\sigma(\mathcal{C}_E,\mathcal{C}_E^{\times})$-continuous.

\begin{theorem}\label{t6}
Let $E \subset c_0$ be a  Banach symmetric sequence space  with Fatou property, and let
 $T$ be a Hermitian operator acting in  $\mathcal{C}_E$ . Then
$T$ is $\sigma(\mathcal{C}_E,\mathcal{C}_E^{\times})$-continuous.
\end{theorem}

\begin{proof} Consider the non-negative continuous
function $\alpha(t)=\|e^{itT}-I\|_{\mathcal B(\mathcal{C}_E)}$, where  $t\in\mathbb R$. As
$\alpha(0)=0$, it follows that there exists $t_0\in\mathbb R$ such that $\alpha(t_0)< 1$. Since the operator
$T$ is Hermitian, it follows that the operator $V=e^{it_0T}$ (and hence $V^{-1}=e^{-it_0T}$) is an isometry on $\mathcal{C}_E$.  By Proposition \ref{p8}, $S=V-I$ is a $\sigma(\mathcal{C}_E,\mathcal{C}_E^{\times})$-continuous operator. In addition,  $\|S\|=\alpha (t_0)<1$.  Now, since
$$
it_0T=ln(I+S)=\sum\limits_{n=1}^{\infty} \frac{(-1)^{n-1} S^n}{n},
$$
Proposition \ref{p4} implies that  $it_0T$ \ is a $\sigma(\mathcal{C}_E,\mathcal{C}_E^{\times})$-continuous operator.
Therefore, $T$ is a $\sigma(\mathcal{C}_E,\mathcal{C}_E^{\times})$-continuous operator.
\end{proof}

If $a,b$ are self-adjoint operators in $\mathcal B(\mathcal H)$, $x\in \mathcal{C}_E$, and
\begin{equation}\label{e2}
T(x)= ax+xb
\end{equation}
for all $x\in \mathcal{C}_E$, then $T$ is a Hermitian operator acting in  $\mathcal{C}_E$ \cite{sourour}. In the following  theorem, utilizing the method of proof of Theorem 1 in \cite{sourour}, we show that every Hermitian operators
acting in a perfect norm ideal $(\mathcal{C}_E, \|\cdot\|_{\mathcal{C}_E}) \neq C_2$  has the form (\ref{e2}).

\begin{theorem}\label{t7}
Let $E \subset c_0$ be a  Banach symmetric sequence space with Fatou property, \ $E \neq l_2$, and let  $T$ be a Hermitian operator acting in a Banach symmetric ideal  $\mathcal{C}_E$. Then there are self-adjoint operators $a,b\in \mathcal B(\mathcal H)$
such that $T(x)= ax+xb$ for all $x\in \mathcal{C}_E$.
\end{theorem}

\begin{proof} As in the proof of Theorem  1 from \cite{sourour}, we have that there are self-adjoint operators
$a,b\in \mathcal B(\mathcal H)$ such that $T(x)= ax+xb$ for all
$x\in \mathcal{F} (\mathcal H)$. Fix $x\in \mathcal{C}_E$. By Proposition \ref{p3},  there is a sequence $\{x_n\}\subset
\mathcal{F} (\mathcal H)$ such that $x_n\stackrel{\sigma(\mathcal{C}_E,\mathcal{C}_E^\times)}{\longrightarrow} x$. If
$y\in \mathcal{C}_E^\times$, then
$$
tr(y(ax_n+ x_n b)) = tr((ya)x_n) + tr((by)x_n)\to tr((ya)x)+tr((by)x)=
$$
$$
=tr(y(ax)+tr(y(xb)))=tr(y(ax+xb)).
$$
Therefore
$$
ax_n+x_n b \stackrel{\sigma(\mathcal{C}_E,\mathcal{C}_E^\times)}{\longrightarrow} ax+xb.
$$
Since $T$ is a Hermitian operator, it follows that $T$ is $\sigma(\mathcal{C}_E,\mathcal{C}_E^\times)$-continuous
(see Theorem \ref{t6}). Therefore
$$
 T(x)=\sigma(\mathcal{C}_E,\mathcal{C}_E^\times)-\lim\limits_{n\to\infty}
T(x_n)=\sigma(\mathcal{C}_E,\mathcal{C}_E^\times)-\lim\limits_{n\to\infty}
(ax_n+x_n b)= ax+xb.
$$
\end{proof}

\section{Isometries of  a  Banach symmetric ideal}

In this section, we prove our main result, Theorem \ref {t65}. The proof of Theorem \ref {t65} is similar to  the proof of Theorem 2 in \cite{sourour}. We use a version of Theorem 1 in \cite{sourour} for a perfect Banach symmetric ideal $\mathcal{C}_E$ (Theorem \ref{t7}) as well as  $\sigma(\mathcal{C}_E,\mathcal{C}_E^{\times})$-continuity of every  isometry on $\mathcal{C}_E$ (Proposition \ref{p8}) and $\sigma(\mathcal{C}_E,\mathcal{C}_E^{\times})$-density of the space $\mathcal F(\mathcal H)$ in $\mathcal C_E$ (Proposition \ref{p3}).

Recall that $x^t $  stands for the transpose of an operator $x \in \mathcal K(\mathcal H)$ with respect to a fixed orthonormal basis
in $\mathcal H$.

\begin{theorem} \label{t65}
Let $E \subset c_0$ be a  Banach symmetric sequence space with Fatou property, \ $E \neq l_2$, and let $V$ be a surjective linear isometry on the Banach symmetric ideal  $\mathcal C_E$.
Then there are  unitary operators $u$ and $v$ on $\mathcal H$ such that
\begin{equation}\label{e15}
V(x) = uxv \ \ (\text{or} \ \ V(x) = ux^tv)
\end{equation}
for all $x \in \mathcal C_E$
\end{theorem}
Note that   each linear operator of the form (\ref{e15}) is an isometry on every Banach symmetric ideal  $\mathcal C_E$.
\begin{proof}
Let  $y\in \mathcal  B(\mathcal  H)$, and let  $l_y(x)=yx$ (respectively, $r_y(x)=xy$) for all  $x\in \mathcal C_E$. It is clear that \ $l_y$ \ and \ $r_y$ \ are  bounded linear operators acting in $\mathcal C_E$. Using Theorem \ref{t7} and repeating the proof of the Theorem 2 \cite{sourour}, we conclude that  there are unitary operators $u, v \in   \mathcal  B(\mathcal  H)$ such that $V l_yV^{-1} = l_{uyu^*}$ and $V r_yV^{-1} = r_{v^*yv}$ \ for any $y\in\ \mathcal  B(\mathcal  H)$.

In the case $V l_yV^{-1} = l_{uyu^*}$,  we define an isometry $V_0$ on $\mathcal C_E$ by the equation $V_0 (x)=u^* V(x) v^*$.
As in the proof of Theorem 2 \cite{sourour}, we get that $V_0(x) = \lambda x$ for every $x \in \mathcal F(\mathcal H)$ and some $\lambda \in \mathbb C$. Since the isometry $V_0$ is $\sigma(\mathcal{C}_E,\mathcal{C}_E^{\times})$-continuous
(Proposition \ref{p8}) and the space $\mathcal F(\mathcal H)$ is $\sigma(\mathcal{C}_E,\mathcal{C}_E^{\times})$-
 dense in $\mathcal C_E$ (Proposition \ref{p3}), it follows that $V_0(x) = \lambda x$ for every $x \in \mathcal C_E$.
Now, since $V_0$ is an  isometry on $\mathcal C_E$, we have $\lambda =1$, i.e.
$V(x) = uxv$ for all $x \in \mathcal C_E$.

In the case $V l_yV^{-1} = r_{v^*yv}$, as  in the proof of Theorem 2 \cite{sourour}, we use the above to derive
that there are unitary operators $u, v \in  \mathcal  B(\mathcal  H)$ such that $V(x) = ux^tv$ for all $x \in \mathcal C_E$.
\end{proof}

\begin{corollary}\label{c2}
Let $E \subset c_0$ be a  Banach symmetric sequence space with Fatou property, \ $E \neq l_2$, and let $V$ be a surjective linear isometry on $\mathcal C_E$. Then
 $V(yx^*y) = V(y) (V(x))^* V(y)$ for all $x,y\in \mathcal C_E$.
\end{corollary}
\begin{proof} By Theorem  \ref{t65}, there are unitary operators $u$ and $v$ on $\mathcal H$ such that $V(x) = uxv \ \ \text{or} \ \ V(x) = ux^tv$
for all $x \in \mathcal C_E$. If $V(x) = uxv$, then
$$V(y) (V(x))^* V(y) =
(uyv)(uxv)^*(uyv)=uyx^*yv=V(yx^*y), \ x, y \in \mathcal C_E.
$$
Since  $(x^t)^*=(x^*)^t, \ x \in \mathcal B(\mathcal H)$, in the case $V(x) = ux^tv$, we have
$$V(y) (V(x))^* V(y) =
(uy^tv)(ux^tv)^*(uy^tv)=
$$
$$
=uy^t(x^*)^ty^tv= u(y(x^*)y)^tv=V(yx^*y), \ x, y \in \mathcal C_E.
$$
\end{proof}

Let  $(X,\|\cdot\|_X)$ be an arbitrary complex Banach space. A surjective (not necessarily linear) mapping   $T\colon X \to X$ is called a surjective 2-local isometry \cite{molnar}, if
for any $x, y \in X $ there exists a surjective linear isometry  $V_{x, y}$ on  $ X $ such that $T(x) = V_{x, y}(x)$ and
$T(y) = V_{x, y}(y) $. It is clear that every surjective linear isometry on $X$ is automatically a surjective 2-local isometry
on $ X $. In addition, 
$$
T(\lambda x) = V_{x, \lambda x}(\lambda x) = \lambda V_{x, \lambda x}(x) = \lambda T(x) $$
for any $x\in X$ and $\lambda \in \mathbb C$.
In particular,
\begin{equation}\label{e25}
T (0) = 0.
\end{equation}
Thus, in order to establish linearity of a 2-local isometry $ T $, it is sufficient to show that
$ T(x + y) = T(x) + T(y) $ for all $ x, y \in X $.

Note also that
\begin{equation}\label{e35}
 \|T (x) -T (y)\|_X = \|V_{x, y}(x) - V_{x, y}(y)\|_X = \| x - y \|_X
\end{equation}
for any $ x, y \in X $.
Therefore, in the case a real Banach space $X$, \ from (\ref{e25}), (\ref{e35}) and Mazur-Ulam Theorem (see, for example, \cite [Chapter I, \S 1.3, Theorem 1.3.5.] {flem}) it follows 
that every surjective 2-local isometry on $ X $ is a linear. For  complex Banach spaces, this fact is not valid.

Using the description of isometries on  a minimal Banach symmetric ideal $\mathcal{C}_E$ from  \cite{sourour},  L. Molnar proved that every surjective 2-local isometry on a  minimal Banach symmetric ideal $\mathcal{C}_E$  is necessarily linear \cite[Corallary 5]{molnar}.

The following Theorem is a version of Molnar's result for a perfect Banach symmetric ideal.
\begin{theorem} \label{t75}
Let $E \subset c_0$ be a  Banach symmetric sequence space with Fatou property, \ $E \neq l_2$, and let $V$ be a surjective 2-local isometry on a Banach symmetric ideal  $\mathcal C_E$.
Then $V$ is a linear isometry on  $\mathcal C_E$.
\end{theorem}
\begin{proof}
Fix $x,y\in \mathcal F(\mathcal H)$ and let  $V_{x,y}\colon \mathcal C_E \to \mathcal C_E$ be a surjective  isometry such that
$V(x)=V_{x,y}(x)$ \ and \ $V(y)=V_{x,y}(y)$. By Theorem \ref{t65}, \ there are  unitary operators $u$ and $v$ on $\mathcal H$ such that $V(x) = uxv$ \ or  \ $V(x) = ux^tv$ (respectively, $V(y) = uyv$ \ or  \ $V(y) = uy^tv$).
Then we have
$$tr(V(x)(V(y))^*)=tr(V_{x,y}(x)(V_{x,y}(y))^*) = tr(xy^*)$$
for every $x,y\in \mathcal F(\mathcal H)$. In addition, $V(x) \in \mathcal F(\mathcal H)$ and $V$ is a bijective mapping
on $ \mathcal F(\mathcal H)$.

If $x,y,z\in \mathcal F(\mathcal H)$, then
$$tr(V(x+y)(V(z))^*)=tr((x+y)z^*), \ \  tr(V(x)V(z)^*)=tr(xz^*),
$$
$$tr(V(y)V(z)^*)=tr(yz^*).$$
Therefore
$$tr((V(x+y)-V(x)-V(y))(V(z))^*)=0$$
for all  $z\in\mathcal F(\mathcal H)$.
Taking $z=x+y$, \ $z=x$ and $z=y$, we obtain
$$tr((V(x+y)-V(x)-V(y)((V(x+y)-V(x)-V(y))^*)=0,
$$
that is, $V(x+y)=V(x)+V(y)$ for all  $x,y\in \mathcal F(\mathcal H)$.
Hence  $V$ is a bijective linear isometry on a normed space  $(\mathcal F(\mathcal H), \|\cdot\|_{\mathcal C_E})$.

Let $\mathcal I$  be the closure of the subspace $\mathcal F(\mathcal H)$ in the Banach space $(\mathcal C_E, \|\cdot\|_{\mathcal C_E})$. It is clear that $(\mathcal I, \|\cdot\|_{\mathcal C_E})$ is a minimal Banach symmetric ideal. Since $V$ is a surjective linear isometry on $(\mathcal F(\mathcal H), \|\cdot\|_{\mathcal C_E})$, it follows that $V$ is a surjective linear isometry on $(\mathcal I, \|\cdot\|_{\mathcal C_E})$.

By Theorem  2 \cite{sourour}, there are  unitary operators $u_1$ and $v_1$ on $\mathcal H$ such that $V(x) = u_1xv_1$ \ or \ $V(x) = u_1x^tv_1$ for all $x \in \mathcal I$. Repeating the ending of the proof of Theorem 1 in \cite{molnar}, we
conclude that  $V$ is a linear isometry on  $\mathcal C_E$.
\end{proof}

\end{document}